\documentclass{amsart}
\usepackage{amsfonts,amscd}

\newtheorem{claim}{}[section]
\newtheorem{theorem}[claim]{Theorem}
\newtheorem{lemma}[claim]{Lemma}
\newtheorem{proposition}[claim]{Proposition}
\newtheorem{corollary}[claim]{Corollary}
\newtheorem{remark}[claim]{Remark}
\newtheorem{definition}[claim]{Definition}
\newtheorem{example}[claim]{Example}

\def\proclaim #1. #2\par{\medbreak
\noindent{\bf#1.\enspace}{\sl#2}\par\medbreak}
\makeatother

\begin{document}

\title{Rigged modules II: multipliers and duality}

\author{David P. Blecher} 
\address{Department of Mathematics, University of Houston, Houston, TX  77204-3008}
\email{dblecher@math.uh.edu}





\begin{abstract}
In  a previous paper with Kashyap  we generalized the theory of $W^*$-modules to the setting of modules over nonselfadjoint dual
operator algebras on a Hilbert space, obtaining the class of weak$^*$-rigged modules.  The present paper and its contemporaneous predecessor comprise  the sequel which we
promised at that time would be forthcoming.   We give many new results about rigged and weak$^*$-rigged modules and their tensor products,
such as an Eilenberg-Watts type theorem.
\end{abstract}

\maketitle

\section{Introduction}   {\em Rigged modules} over a (nonselfadjoint) 
operator algebra are the generalization from \cite{DB2,BMP} of the important  class of modules over
$C^*$-algebras known as {\em Hilbert $C^*$-modules}.
A {\em $W^*$-module} is a Hilbert $C^*$-module over a von Neumann algebra which is `selfdual' (see e.g.\ \cite{Pas,Bsd}), or equivalently 
which has a Banach space predual (a result of Zettl, see e.g.\ \cite[Corollary 3.5]{BM} for one proof of this).  
 The {\em weak$^*$-rigged} or  $w^*$-{\em rigged modules}, introduced in \cite{BK2} (see also \cite{BKraus}; or \cite[Section 5]{BM} for an earlier variant), are a generalization of $W^*$-modules to the setting of modules over a (nonselfadjoint) dual
operator algebra.    By the latter term we mean a unital weak* closed
algebra of operators on a Hilbert space.

In \cite{BK2}  we generalized basic aspects of the theory of $W^*$-modules, and this may be seen also as the weak* variant of the theory of rigged modules from \cite{DB1} (see also \cite{BMP}). 
 The present paper and its contemporaneous predecessor comprise  the sequel which we 
promised at the time of \cite{BK2} would be forthcoming.  In the present paper we discuss rigged modules and `correspondences' using the concept of the operator space left multiplier algebra of $Y$ in the sense of 
{\rm \cite[Section 4.5]{BEZ}}.   We also discuss a  connection between rigged and 
weak$^*$-rigged modules, the exterior tensor product, orthogonally  complemented submodules, and other topics such as an Eilenberg-Watts type theorem characterizing
functors between categories of rigged or weak* rigged modules.     

In the course of this work 
we noticed several things that were 
 missed, or not stated (or proved), or which could be simplified,
from the time the earlier work on rigged and weak$^*$-rigged modules
was done.   We take the opportunity to correct/present/simplify these things here.

Turning to background, we will use the notation from \cite{BK1,BK2,UK1,BK3}, and perspectives from \cite{DB2}. We will assume that the
reader is familiar with basic notions from operator space theory which may be found in any current text on that subject. The reader may
consult \cite{DBbook} as a reference for any other unexplained terms
here. We assume that the reader is familiar with basic Banach space and operator space duality principles such as the Krein-Smulian Theorem. 
We often abbreviate `weak$^*$' to `$w^*$'.    
 A right dual operator $M$-module is a nondegenerate $M$-module $Y$, which is also a dual operator space, such that the module action is completely contractive and separately weak* continuous. We use standard notation for module mapping spaces; e.g.\ $CB(X, N)_N$ (resp.\ $CB^\sigma(X, N)_N$) are the completely bounded (resp.\ and weak* continuous) right $N$-module maps from $X$ to $N$.  We often use the {\em normal module Haagerup tensor product} $Y \otimes^{\sigma h}_{M} Z$, and its universal property from \cite{EP}, which loosely says that it ‘linearizes 
completely contractive $M$-balanced separately weak* continuous bilinear maps' (balanced means that 
$u(xa,y) = u(x,ay)$ for $a \in M$).  We assume that the reader is familiar with the notation and facts about this tensor product from \cite[Section 2]{BK1}.    Although we shall not use it here, in passing we remark that the module tensor product facts in that 
section work even without assuming all of the constituents of the definition of $Y$ being a dual operator $M$-module,
so long as it is a dual operator space and an $M$-module.   For any operator space $X$ we write $C_n(X)$ for the column space of $n \times 1$ matrices with entries in $X$, with its canonical norm from operator space theory.

\begin{definition} \cite{BK2} \label{wrig} Suppose that $Y$ is a dual operator
space and a right module over a dual operator algebra $M$. Suppose
that there exists a net of positive integers $(n(\alpha))$, and
$w^*$-continuous completely contractive $M$-module maps
$\phi_{\alpha} : Y \to C_{n(\alpha)}(M)$ and $\psi_{\alpha} :
C_{n(\alpha)}(M) \to Y$, with $\psi_{\alpha}( \phi_{\alpha}(y))$ converging
to $y$ in the weak* topology on $Y$, for all $y \in Y$.   Then we say
that $Y$ is a {\em right $w^*$-rigged module} (or {\em weak$^*$-rigged module}) over $M$.
\end{definition}

We remark that the fact that $w^*$-rigged modules are dual operator modules seems not to have been proved  in the development in Section 2 and the 
start of Section 3 in \cite{BK2}
but seemingly assumed in the proof.   We give a proof of this early on \cite{BK3}.

As on p.\ 348 of \cite{BK2}, the operator space structure of a 
$w^*$-rigged module $Y$ over  $M$ is determined by
$\Vert [y_{ij}] \Vert_{M_n(Y)}
= \sup_\alpha \; \Vert [\phi_{\alpha}(y_{ij})] \Vert$ for $[y_{ij}] \in M_n(Y).$

The {\em 
rigged modules} of {\rm \cite{DB2}} may be defined similarly 
to Definition \ref{wrig}, but with the words `dual' and `$w^*$-continuous'
removed, and the weak* topology replaced by the norm topology, and $M$ now an
approximately unital operator algebra.   
This simpler reformulation of the definition of a rigged module, and its equivalence
with the definitions in \cite{DB2}, may be found in \cite[Section 3]{BHN}.   The 
operator space structure of a rigged module $Y$ over  $M$ is determined by the same formula
as at the end of the last paragraph, but for the appropriate $\phi_{\alpha}$ in this case.
See \cite{DB2} for details.

We say that $w^*$-rigged modules are {\em unitarily isomorphic} if there exists a completely 
isometric surjective weak* homeomorphic module map between them.   Similarly for
rigged modules, with of course `weak* homeomorphic' dropped.

Every right $w^*$-rigged module  (resp.\ right rigged module) $Y$ over $M$ gives rise to a canonical left $w^*$-rigged (resp.\ left rigged module) $M$-module
$\tilde{Y}$, and a pairing $(\cdot, \cdot) : \tilde{Y} \times Y \to M$ (see \cite{BK2,DB2}).
Indeed in the $w^*$-rigged  case, $\tilde{Y}$ turns out to be completely isometric to $CB^{\sigma}(Y,M)_{M}$ as dual operator 
$M$-modules, together with its canonical pairing with $Y$.   We also have $\widetilde{\tilde{Y}} = Y$.
The morphisms between $w^*$-rigged $M$-modules are the {\em adjointable} $M$-module maps \cite{BK2,BK3}, which 
turn out to coincide with the 
weak* continuous completely bounded  $M$-module maps (see \cite[Proposition 3.4]{BK2}).   
We write $\mathbb{B}(Z,W)$ for the weak* continuous completely bounded  $M$-module maps
from a  $w^*$-rigged $M$-module $Z$ into a dual operator $M$-module $W$, with as
usual $\mathbb{B}(Z) = \mathbb{B}(Z,Z)$.     We also use this notation for the {\em adjointable maps}
between rigged modules \cite{DB2}.     We write $\mathbb{K}(Y)_A$ for the  so-called {\em compact adjointable}
right $A$-module maps on a right $A$-rigged module $Y$; namely the closure of the 
span of the 
maps on $Y$ of form $y \mapsto y' (x,y)$ for some  $y' \in Y$ and $x \in \tilde{Y}$ (see {\rm \cite{DB2}}).

\section{Rigged modules, multipliers, correspondences, and duality}

The following important facts about rigged modules 
do not appear to be in the literature:

\begin{lemma} \label{notin}
If $Y$ is a rigged module over an
operator algebra $A$, viewed as an operator space, and if ${\mathcal M}_{\ell}(Y)$ is the 
operator space left multiplier algebra of $Y$ in the sense of 
{\rm \cite[Section 4.5]{BEZ}}, then 
${\mathcal M}_{\ell}(Y) = CB(Y)_A$ completely isometrically isomorphically.
This also equals the left multiplier algebra of 
$\mathbb{K}(Y)_A$, where the latter is the compact adjointable
maps on $Y$.
\end{lemma}

\begin{proof} 
It is known (see e.g.\ \cite[Theorem 3.6]{DB2}) that $\mathbb{K}(Y)_A$ is a left ideal in
$CB(Y)_A$.  This gives a map $CB(Y)_A \to LM(\mathbb{K}(Y)_A)$.
Conversely, since $Y$ is a left operator 
$\mathbb{K}(Y)_A$-module (by the same cited theorem), it is a
left operator $LM(\mathbb{K}(Y)_A)$-module by 3.1.11 in \cite{DBbook}.
Hence we obtain a completely contractive homomorphism 
$LM(\mathbb{K}(Y)_A) \to CB(Y)_A$.  It is easy to argue that these
maps are mutual inverses, so that $CB(Y)_A \cong
LM(\mathbb{K}(Y)_A)$ completely isometrically isomorphically.  (This argument may have originally
been due to Paulsen.)  

By facts in the theory of operator space multipliers (see e.g.\ \cite[Theorem 4.5.5]{DBbook}), 
 the `identity map'  is a completely contractive homomorphism ${\mathcal
M}_{\ell}(Y) \to CB(Y)$.   This maps into $CB(Y)_A$, since for example right 
multiplication by $a \in A$ is easily seen to be  a right operator space multiplier, and left and right 
 operator space multipliers commute (see 4.5.6 in \cite{DBbook}).     From \cite{DB2}   
we know that  $CB(Y)_A$ is an operator algebra. 
Also, by the last paragraph $Y$ is a left operator
 $CB(Y)_A$-module (with the canonical action).
By the  operator space multiplier theory (see e.g.\ \cite[Theorem 4.6.2 (1) and (2)]{DBbook})
there exists a completely contractive
homomorphism $\pi : CB(Y)_A \to {\mathcal M}_{\ell}(Y)$ with
$\pi(T)(y) = T(y)$ for all $y \in Y, T \in CB(Y)_A$.  That is,
$\pi(T) = T$.  Thus $CB(Y)_A = {\mathcal M}_{\ell}(Y)$.  
\end{proof}  

The last result should have many consequences.  In the remainder of this section we give several.

\begin{corollary}  \label{lms}   For any orthogonally complemented (in the sense of {\rm \cite[Section 7]{DB2}}) submodule $W$ of 
a rigged module $Y$ over an
operator algebra $A$,  there is
a unique contractive linear projection from $Y$ onto $W$.  
The right $M$-summands 
in the sense of
{\rm \cite{BEZ}} (see also {\rm \cite[Sections 4.5 and 4.8]{DBbook}})
in  such $Y$ 
   are precisely the orthogonally complemented submodules of $Y$.
 \end{corollary}

\begin{proof} 
 The orthogonal projections in ${\mathcal M}_{\ell}(Y)$, which 
by the previous result are the completely contractive idempotents
 in $CB(Y)_M$,
 are the left $M$-projections on $Y$ by \cite{BEZ},
and the right $M$-summands are their ranges.
 These ranges are just the orthogonally complemented submodules.

 The first assertion  is a general
fact about right $M$-summands of an operator space from \cite{BEZ}. 
\end{proof}

An early  prototype of the 
$w^*$-rigged  modules appeared in  \cite[Section 5]{BM}.
 We now connect these two notions.
Some examples of the modules characterized here may be found 
e.g.\ in \cite[p.\ 405]{BM}.    For example every $W^*$-module (defined in the first line of our 
paper) satisfies these conditions.  

\begin{theorem}  \label{bmex}
Let $Y$ be a rigged module over
a dual operator algebra $M$. 
Suppose that 
$Y$ has a predual operator space, and that $(x,\cdot)$ is weak* continuous
for all $x \in \tilde{Y}$.
  Then $Y$ is a $w^*$-rigged module,
and $Y$ is self-dual (that is,
$CB(Y,M)_M \cong \tilde{Y}$ via the canonical map), and 
$CB(Y)_M = CB^\sigma(Y)_M = \mathbb{B}(Y)_M$.
Thus $Y$ belongs to the class of modules considered
in Lemma {\rm 5.1} and Corollaries {\rm 5.2} 
and {\rm 5.5} in {\rm \cite{BM}}, and therefore 
satisfies all the conclusions of those results.
   \end{theorem}

\begin{proof}   
By Lemma \ref{notin} and 
because left multipliers on a dual space 
are known to be weak* continuous \cite[Theorem 4.7.1]{DBbook},
 we have $$CB(Y)_M = {\mathcal M}_{\ell}(Y) = CB^\sigma(Y)_M .$$
Given a bounded net $m_t \to m$ weak* in $M$, suppose that a subnet
$y m_{t_{\nu}} \to y'$ weak* in $Y$.   Then 
$(x,y m_{t_{\nu}}) \to (x,y')$ for all $x \in \tilde{Y}$.  However it also converges to $(x,y)m$, and so
$(x,y'-ym) = 0$.   It follows that $y' = ym$, so that by topology $y m_t \to ym$ weak*.
So the map  $m \mapsto ym$ is  weak* continuous by the Krein-Smulian theorem.  

It follows that in the definition of $Y$ being a
rigged module (below Definition \ref{wrig})
we may assume that the maps $\phi_{\alpha}, \phi_{\alpha}$
are weak* continuous.   Indeed in \cite{DB2} the `coordinates'
of $\phi_{\alpha}$ are usually assumed to be of the form $(x,\cdot)$ for $x \in \tilde{Y}$,
hence are weak* continuous.   For $\psi_{\alpha}$
this follows from the fact proved in the last paragraph.  So $Y$ is a $w^*$-rigged module.

Applying the relation at the end of the first paragraph of the proof
 to the direct column sum
$Y \oplus^c M$  (see \cite{BK2}) we have
$CB(Y \oplus^c M)_M$ equal to $CB^\sigma(Y \oplus^c M)_M$,
from which it is clear that $$CB(Y,M)_M
= CB^\sigma(Y,M)_M \cong \tilde{Y}.$$   So $Y$ is self-dual.  
The other conclusions are easy.
    \end{proof}

The last result may be viewed as a `nonselfadjoint variant' of the result of 
Zettl mentioned in the first lines of the paper.

\begin{remark}    {\rm   The condition in the theorem that $(x,\cdot)$ is weak* continuous
may be automatic, although to get this one may need to assume that 
$b \mapsto yb$ is weak* continuous on $M$ for each $y \in Y$.   We were able to show without this
$(x,\cdot)$ condition
that  $Y \otimes_{hM} H^c$  and $Y \otimes^{\sigma h}_M H^c$ are Hilbert spaces, and if these 
two spaces coincide then  the conclusions of the theorem hold.   We were also able to 
prove the theorem with the weak* continuity assumption on $(x,\cdot)$ replaced by the 
weak* continuity of  $b \mapsto yb$ condition,  if $M$ acts faithfully 
on the right on $Y$ (that is there is an unique $b \in M$ with $Y b = 0$).   To see this, let 
$f \in CB(Y,M)_M$, let $y_t \to y$ be a bounded weak* convergent net in $Y$,
 and let $y_0 \in Y$ be fixed.
By the first  paragraph of the proof the map $y \mapsto y_0 \, f(y)$ is weak* continuous on $Y$, so 
$y_0 \, f(y_t) \to y_0 \, f(y)$.   Suppose that we have a weak* convergent subnet
$f(y_{t_\nu}) \to b$ in $M$.   Then $y_0 \, f(y_{t_\nu}) \to y_0 \, b$ weak*.
Thus $y_0 \, b = y_0 \, f_k(y)$ for all  $y_0 \in Y$, and we deduce that $b = f_k(y)$.   By topology
$f(y_t) \to f_k(y)$ weak*.    Hence by  the Krein-Smulian theorem
$f$ is weak* continuous.    It follows as in the proof that 
we may assume that the maps $\phi_{\alpha}, \phi_{\alpha}$
are weak* continuous, and $Y$ is $w^*$-rigged.   Hence 
$CB(Y,M)_M = CB^\sigma(Y,M)_M \cong \tilde{Y}$, so $Y$ is self-dual, and we may continue as before.}
\end{remark}

Recall that an {\em approximately unital operator algebra}
is one which has a contractive approximate identity.
 
\begin{theorem}  \label{ncor}  Suppose that $A, B$ are 
 approximately unital operator algebras, and that $Y$ is a 
right rigged $B$-module which is a nondegenerate
left $A$-module via a homomorphism $\theta : A \to \mathbb{B}(Y)_B
= M(\mathbb{K}(Y)_B)$.
Then with this action  $Y$ is a left operator $A$-module if and only if 
$\theta$ is completely contractive.    If these hold then 
 $\theta$ is essential in the sense of {\rm \cite[p.\ 400-401]{DB2}}. 
 In particular, there is a contractive approximate identity $(e_t)$
for $A$ with $e_t y \to y$ and $x e_t \to x$ for all $y \in Y, x \in \tilde{Y}$.
\end{theorem}

\begin{proof}   The first assertion may be seen for example from Lemma \ref{notin} and the 
fact that the left operator $A$-module actions on 
$Y$ are in bijective correspondence with the completely contractive homomorphisms
into ${\mathcal M}_{\ell}(Y)$ which give a nondegenerate left module action on $Y$.   The one direction of this
 follows from e.g.\ \cite[Theorem 4.6.2 (1) and (2)]{DBbook}.   The other direction follows from 
3.1.12 in \cite{DBbook} and the fact that  any operator space $Y$ is  
a left operator ${\mathcal M}_{\ell}(Y)$-module (by \cite[Theorem 4.5.5]{DBbook}).

Viewing $M(\mathbb{K}(Y)_B) \subset (\mathbb{K}(Y)_B)^{**}$,
we have that $\theta$ extends uniquely to a 
completely contractive homomorphism $\tilde{\theta} : 
A^{**} \to (\mathbb{K}(Y)_B)^{**}$ by 2.5.5 in \cite{DBbook}.
Since $\theta(e_t) z \to z$ for all $z \in \mathbb{K}(Y)_B$ and
contractive approximate identity $(e_t)$ of $A$, it follows that any weak* limit point $\eta$ of $(\theta(e_t))$ satisfies $\eta z = z$.  So $\eta$ is
a left identity for $(\mathbb{K}(Y)_B)^{**}$, hence equals the
identity $1$ for that algebra (see \cite[Proposition 2.5.8]{DBbook}).  
 So $\theta(e_t) \to 
1$ weak*, by topology.   Then $z \theta(e_t) \to z$ weak*  in $ \mathbb{K}(Y)_B^{**}$ for $z \in \mathbb{K}(Y)_B$, and hence 
weakly in $\mathbb{K}(Y)_B$ (note that $\mathbb{K}(Y)_B$ is an ideal in $\mathbb{B}(Y)_B$). 
By Mazur's theorem, taking convex combinations we get a norm
bounded net satisfying \cite[Proposition 6.2 (2)]{DB2}.
So $\theta$ is essential.  The last assertion follows from 
\cite[Proposition 6.3]{DB2}.    
   \end{proof} 

A bimodule satisfying the conditions in the last result
will  be called a (right) $A$-$B$-{\em correspondence}. The last theorem  shows that the original 
definition
in \cite[Proposition 6.3]{DB2} can be substantially simplified.

The interior tensor product of right rigged modules  from \cite[p.\ 400-401]{DB2}
is simply the module Haagerup tensor product (see \cite{BMP,DBbook}) of a 
right $A$-rigged module and a right $A$-$B$-correspondence.
We will write this tensor product as 
$Y \otimes_\theta Z$, where $\theta$ is the left action as above.
However we will not focus much on rigged modules in this paper, since that theory is
older and more developed.  

We will use later the interior tensor product of weak* rigged modules \cite{BK2,BK3}.  Here 
 $Y$ is a right $w^*$-rigged module over a dual operator algebra $M$
and, that $Z$ is a right $w^*$-rigged module over 
a dual operator algebra $N$, and $\theta :M \to \mathbb{B}(Z)$
is a weak* continuous unital completely contractive homomorphism. Because $Z$ is a left 
operator module $\mathbb{B}(Z)$-module
(see p.\ 349 in  \cite{BK2}), $Z$ becomes an essential
left dual operator module over $M$ under 
the action $m \cdot z = \theta(m) z$. In this case we say $Z$ is a right $M$-$N$-{\em correspondence}
(an abusive notation because this concept is the weak* variant of the analoguous notion
studied earlier in this section under the same name).
We form the normal module Haagerup tensor product  $Y \otimes^{\sigma h}_{M} Z$
which we also write as $Y \otimes_{\theta} Z$ (again a somewhat abusive notation;
the context will have to make it clear whether we are using the rigged or the 
$w^*$-rigged variant).  By 3.3 in \cite{BK2} this a right $w^*$-rigged module
over $N$, called  the {\em interior tensor product} of $w^*$-rigged modules.

\section{Eilenberg-Watts type theorem}

The norm on the matrix space $M_{m,n}(CB(Y,Z)_M)$ (and on its subspace $M_{m,n}(\mathbb{B}(Y,Z))$)
is the operator space one,  
namely giving $[f_{ij}]$  the `completely bounded norm' in $CB(Y,M_{m,n}(Z))$ of the map $y \mapsto [f_{ij}(y)]$. 
We write this norm as $\Vert [f_{ij}] \Vert_{cb}$.

\begin{lemma} \label{lll}   Suppose that $Y$ is a right $w^*$-rigged (resp.\ rigged) module,
and $Z$ is a right dual  operator module (resp.\ right operator module), over 
a dual operator algebra (resp.\ operator algebra) $M$.    
For  $m, n \in \mathbb{N}$ suppose that $[f_{ij}] \in M_{m,n}(CB(Y,Z)_M)$, with 
each $f_{ij}$ weak* continuous in the $w^*$-rigged case.     Then
$$\Vert [f_{ij}] \Vert_{cb} \, = \, \sup_\alpha \,  \Vert [ f_{ij}(y_k^\alpha)] \Vert,$$
where  $[ f_{ij}(y_k^\alpha)]$ is indexed on rows by $i$ and on columns by $j$ and $k$,
and where $(y_k^\alpha)$ are the `coordinates' of the map $\psi_\alpha$ in Definition
\ref{wrig} (so that $\psi_\alpha([b_k]) = \sum_k \, y_k^\alpha b_k$).  
Also this norm also equals the `completely bounded norm' in $CB(C_n(Y),C_m(Z))$ 
of the map
$[y_j] \mapsto [\sum_j \, f_{ij}(y_j)]$ on $C_n(Z)$.   
In particular for $w^*$-rigged modules $Y, Z$ over $M$ 
we have  $$M_{m,n}(\mathbb{B}(Y,Z)) \cong
\mathbb{B}(C_n(Y),C_m(Z))$$ completely isometrically.  
\end{lemma}  

\begin{proof}      
The assertions for $w^*$-rigged modules  follow by \cite[Corollary 3.6]{BK2}, or by the weak*  variant of the
following.   In the  rigged module case  the result follows by a trick which occurs very frequently in the 
theory (see e.g.\ \cite{BMP}), so we will be brief.  Write the map $\phi_\alpha$ in Definition
\ref{wrig} as 
$\phi_\alpha(y) = [x_k^\alpha(y)]$, and set $y^\alpha =(y_k^\alpha) \in  M_{1,n}(Y)$.  Then  for $[y_{pq}] \in M_r(Y)$ of norm $1$ we have 
$$[ f_{ij}(y_{pq})] = \lim_\alpha \, [ f_{ij}(\sum_k \, y_k^\alpha (x_k^\alpha , y_{pq})) ]
= \lim_\alpha \, [ \sum_k \,  f_{ij}(y_k^\alpha) (x_k^\alpha , y_{pq})] .$$
The norm of this is dominated by $\sup_\alpha \,  \Vert [ f_{ij}(y_k^\alpha)] \Vert$, which
in turn is  dominated by $\Vert [f_{ij}] \Vert_{cb}$ since $\Vert y^\alpha \Vert \leq 1$.
This proves the displayed equation.   A similar computation shows that 
$\Vert [\sum_j \, f_{ij}(y_j^{pq})] \Vert \leq \sup_\alpha \, \{ \Vert [ f_{ij}(y_k^\alpha)] \Vert$
for a matrix $[y_i^{pq}]$ of norm $1$ with entries $y_i^{pq}$ in $Y$ indexed on rows by $i,p$ and
on columns by $q$.      In turn $\Vert [ f_{ij}(y_k^\alpha)] \Vert$  is  dominated by
the completely bounded norm in $CB(C_n(Y),C_m(Z))$, as may be seen by viewing $f$ 
as `acting by left multiplication' on 
the $n \times (n \cdot n(\alpha))$ matrix $y^\alpha \otimes I_n$.  \end{proof}

For a dual operator algebra $M$ let $\mathcal{W}_{M}$ denote the category
of right $w^*$-rigged modules over $M$.
The morphisms are the weak* continuous (or equivalently,
adjointable) completely bounded
$M$-module maps.   For an approximately unital operator algebra $M$ let $\mathcal{R}_{M}$ be the category
of right rigged modules over $M$, with 
morphisms the  adjointable completely bounded
$M$-module maps. 
  
 We will say that a functor $F$ is completely contractive (resp.\ linear,
normal, strongly continuous) if $T \mapsto F(T)$ is completely contractive 
(resp.\ linear, weak* continuous, takes bounded strongly
convergent (that is, `point-norm' convergent) nets to 
strongly
convergent nets) on the spaces
of morphisms.

\begin{proposition} \label{okba}
For approximately unital 
operator algebras (resp.\ dual operator algebras) $M$ and $N$ let 
$Z$ be a right $M$-$N$-correspondence.
Then the interior tensor product with $Z$ is a 
strongly continuous normal (resp.\ normal) completely contractive linear functor from $\mathcal{W}_{M}$ to $\mathcal{W}_{M}$
 (resp.\ $\mathcal{R}_{M}$ to $\mathcal{R}_{N}$).

In particular, if $M$ and $N$ are weak* Morita equivalent dual operator algebras
in the sense of \cite{BK1}, then their categories of  right $w^*$-rigged modules
are isomorphic. Moreover this isomorphism is implemented by
tensoring with the equivalence bimodule.  
\end{proposition}

\begin{proof} Let $F(Y) = Y \otimes_\theta Z$ be the 
interior tensor product.    That $F$ is completely contractive 
follows from Proposition 2.2   in \cite{BK3} and the remark after it, and it is easily seen to be
a linear functor.   If a bounded net $T_t \to T$ in the strong (resp.\ 
weak*) topology in $\mathbb{B}(Y_1,Y_2)$ then $T_t \otimes I  \to T \otimes I$
strongly (resp.\ weak*, see \cite[Theorem 3.1]{BKraus}).   

 If  $(M, N, X, Y)$ is a weak* Morita context in the sense of  \cite{BK1} then by the above $\mathcal{F}(Z)$ = $Z \otimes^{ \sigma h}_{ M} X$
is a completely contractive normal functor from $\mathcal{R}_{M}$ to $\mathcal{R}_{N}$, with 
`inverse' the 
functor $\mathcal{G}$ from $ \mathcal{R}_{N}$ to
 $\mathcal{R}_{M}$ defined by $\mathcal{G}(W)$ = $W \otimes^{\sigma h}_{N} Y$.
 As in Theorem 3.5 in \cite{BK1}
  $F$ and $G$ are inverse
functors via completely isometric isomorphisms, and so the categories $\mathcal{R}_{M}$ and  $\mathcal{R}_{N}$ are isomorphic.
\end{proof}

\begin{theorem} \label{EW}  Let 
$M$ and $N$ be  approximately unital 
operator algebras (resp.\ dual operator algebras),
and suppose that $F$ is a strongly continuous normal (resp.\ normal) completely contractive linear functor from $\mathcal{W}_{M}$ to $\mathcal{W}_{N}$
 (resp.\ $\mathcal{R}_{M}$ to $\mathcal{R}_{N}$).
Then there exists a  right $M$-$N$-correspondence $Z$ such that $F$
is naturally unitarily  isomorphic to the interior tensor product
with $Z$.  
\end{theorem} \begin{proof}  We are adapting the proof of the 
$C^*$-module variant in 
\cite[Theorem 5.4]{DB1}.   Let $Z = F(M)$.
We first prove that 
$C_n(F(M)) \cong F(C_n(M))$.   The proof of the analogous statement
in \cite{DB1} does not work, instead we proceed 
as follows.   Indeed if $i_k : M \to C_n(M)$ 
and $\pi_k : C_n(M)  \to M$ are the canonical inclusions and
projections, then these are clearly adjointable.
Then $i = (i_k) \in M_{1,n}(\mathbb{B}(M,C_n(M)))$
 as is $\pi = [\pi_k] \in C_n(\mathbb{B}(C_n(M),M))$.
Thus $F(i)$ is in $M_{1,n}(\mathbb{B}(Z,F(C_n(M)))$ is a contraction,
as is $F(\pi) \in C_n(\mathbb{B}(F(C_n(M)), Z))$.
   By Lemma \ref{lll},
we may view $F(i)$ as a contraction in
$\mathbb{B}(C_n(Z),F(C_n(M)))$,
 and  $F(\pi)$ as a contraction in  $\mathbb{B}(F(C_n(M)),C_n(Z))$.
The composition of these latter (complete) contractions in either order
is easily seen to be the identity, so that 
indeed $F(C_n(M)) \cong C_n(Z)$ as desired.
Because we shall need it shortly we note that the unitary morphism
$C_n(Z) \to F(C_n(M))$ here is $[z_k] \mapsto \sum_k \, F(i_k)(z_k)$.

As in \cite[Theorem 5.4]{DB1}, $Z$ is a right rigged (resp.\ w$^*$-rigged)
 module over $N$,
and we make $Z$ into an $M$-$N$-bimodule by defining $m z = F(L_{m})(z)$ for $m \in M, z \in Z$.  Here $L_m: M \to M$ is left multiplication
by $m$, a completely bounded adjointable map.
Since $F$ is completely contractive, it is easy to argue that the 
associated homomorphism  
$\theta : M \to \mathbb{B}(Z)$ is completely contractive. 
 Since $F$ is strongly continuous (resp.\ normal)
the left action of $M$ on $Z$ is nondegenerate
(resp.\ separately weak* continuous),
and $Z$ is a right $M$-$N$-correspondence.  
 
Define a bilinear map $\tau : Y \otimes_\theta Z \to F(Y)$ by 
$(y,z) \mapsto F(L_y)(z)$, where $L_y$ is  left multiplication
by $y$ on $M$.  This map is an $M$-balanced right $N$-module map
as in \cite[Theorem 5.4]{DB1},
and in the $w^*$-rigged case it is clearly separately weak* continuous.
It is completely contractive in
the sense of  Christensen and Sinclair, since if $y = [y_{ij}]
\in {\rm Ball}(M_n(Y)), z = [z_{ij}]
\in M_n(Z)$ then $$[L_{y_{ik}}] \in {\rm Ball}(M_n(\mathbb{B}(M,Y))) =
 {\rm Ball}(\mathbb{B}(M,M_n(Y))),$$ so that 
$[F(L_{y_{ik}})] \in {\rm Ball}(M_n(\mathbb{B}(Z,F(Y)))$.
Since $M_n(\mathbb{B}(Z,F(Y)))$ may be identified with
$\mathbb{B}(C_n(Z),C_n(F(Y)))$ by Lemma \ref{lll},
it is easy to see that
$$\Vert [\sum_k \, F(L_{y_{ik}})(z_{kj}) ] \Vert
\leq \Vert [z_{ij}] \Vert,$$ 
so that  $\tau$ is completely contractive.    
 By the universal property of the tensor product,
we obtain a 
complete contractive $N$-module map $\tau_Y: Y \otimes_\theta Z \to F(Y)$
which is weak* continuous in the $w^*$-rigged case.

That $\tau_Y$ is a complete  isometry 
is similar to the (matrix normed
version of the) computation in \cite[Theorem 5.4]{DB1}.   However to take
into account the $w^*$-rigged module case the argument changes a bit.
In either case, for $u \in Y \otimes_\theta Z$ we have $\tau_Y(u)$ is the appropriate
limit over $\alpha$ of $F(\psi_\alpha) F(\phi_\alpha) \tau_Y(u)$.
It follows that $\Vert (\tau_Y)_n(u) \Vert_{M_n(F(Y))} = 
\sup_\alpha \; \Vert [ F(\phi_\alpha)  \tau_Y(u_{ij})] \Vert$ for $[u_{ij}] \in M_n(Y \otimes_\theta Z)$.
As  at the bottom of p.\ 277 of \cite{DB1} we have
$$F(\phi_\alpha)  \tau_Y(u_{ij}) = \tau_{C_{n(\alpha)}(M)}((\varphi_\alpha \otimes I)(u_{ij})).$$  Since $\tau_{C_{n(\alpha)}(M)}$ is a complete  isometry we deduce that 
$$\Vert (\tau_Y)_n(u) \Vert_{M_n(F(Y))} = 
\sup_\alpha \; \Vert [(\varphi_\alpha \otimes I)(u_{ij}))] \Vert =  \Vert u \Vert_{M_n(Y \otimes_\theta Z)},$$
with the last equality holding by the formula immediately after Definition \ref{wrig}, since $\varphi_\alpha \otimes I$ and $\psi_\alpha \otimes I$
are the asymptotic factorization maps for $Y \otimes_\theta Z$.
Thus  $\tau_Y$ is a complete  isometry. 
   
That $\tau_Y$ has dense range follows similarly to the argument for this in \cite[Theorem 5.4]{DB1}, the key 
point being that the  functions
$\tau_Y \circ (\psi_\alpha \otimes I)$ and $F(\psi_\alpha) \circ \tau_{C_{n(\alpha)}(M)}$
agree on $C_{n(\alpha)}(M) \otimes_M Y$. 
So $\tau_Y$ is a completely isometric isomorphism, that is, a unitary isomorphism, and 
it is an easy exercise to see that it implements the natural equivalence in the desired sense.   
 \end{proof}

\begin{remark}  {\rm  As in pure algebra, it is an easy exercise to see that 
this yields a  bijection between (isomorphism classes of)
right  $M$-$N$-correspondences and (isomorphism classes of)
such  strongly continuous completely contractive functors.
Composition of such functors corresponds to the interior tensor product
of the bimodules.}
\end{remark}

\section{The exterior tensor product of
$w^*$-rigged modules}   If $Y$ is a right  $w^*$-rigged module over $M$, and if $Z$ is a
right $w^*$-rigged  module over $N$, 
we define the {\em weak$^*$-exterior tensor product} $Y \overline
{\otimes}  Z$ to be their {\em normal minimal (or spatial) tensor product} (see e.g.\
1.6.5 in \cite{DBbook}).  We may view it as a module over $M \overline{\otimes} N$  as follows. Let $L(Y)$ and $L(Z)$ be the weak linking algebras for $Y$ and $Z$ respectively (as in 3.2 in \cite{BK2}). Viewing $Y$ and $Z$ as the $1$-$2$ entries of $L(Y)$ and $L(Z)$  respectively, identify 
$Y \otimes Z$ with  the obvious subspace of the dual operator algebra tensor product $L(Y) \overline{\otimes} L(Z)$. Write $Y \overline {\otimes}  Z$ for its completion in the
weak* topology of $L(Y) \overline{\otimes} L(Z)$.  In this way, $Y \overline {\otimes}  Z$ can be seen to be
invariant under right multiplication by the $2$-$2$-corner  $M \overline{\otimes} N$
of $L(Y) \overline{\otimes} L(Z)$.  Thus $Y \overline {\otimes}  Z$  is a right dual operator 
$(M \overline{\otimes} N)$-module.

The normal minimal tensor product of any  dual operator spaces $Y$ and $Z$,
and in particular  hence the  exterior tensor product of
$w^*$-rigged modules, is
completely isometrically and weak*-homeomorphically contained  
in $(Y_* \widehat{\otimes} Z_*)^*$, where $\hat{\otimes}$ is the  operator space projective 
tensor product.
Thus it is contained completely isometrically and weak*-homeomorphically, via the canonical inclusions, in
$CB(Y_*,Z)$ and $CB(Z_*, Y)$.  
Indeed by  basic operator space theory (see e.g.\
\cite{ERbook,DBbook}), we can identify $(Y_* \hat{\otimes} Z_*)^* = CB(Y_*,Z) = CB(Z_*, Y)$ with the normal Fubini
tensor product of $Y$ and $Z$, and it is known that this contains 
a canonical copy of $Y \overline{\otimes} Z$ (see \cite[Theorem 7.2.3]{ERbook}).

In what follows we will use the fact that the normal minimal tensor product is functorial.  That is, if
$Y_k$ and $Z_k$ are dual operator spaces,
and if  $T_k : Y_k \to Z_k$ are  completely bounded weak* continuous 
maps, for $k = 1,2$, then 
$T_1\otimes T_2: Y_1  \overline{\otimes} Z_1 \to Y_2  \overline{\otimes} Z_2$  defines a unique completely bounded weak* continuous  map. Moreover, ${\lVert T_1 \otimes T_2 \rVert}_{cb} \leq {\lVert T_1 \rVert}_{cb} {\lVert T_2 \rVert}_{cb}$.   This also follows from some  basic operator space theory (see e.g.\
\cite{ERbook,DBbook}).
Tensoring the predual maps of $T_k$  with respect to the operator space projective 
tensor product, and then dualizing, gives a weak* continuous 
map $u : ((Y_1)_* \hat{\otimes} (Z_1)_*)^* \to ((Y_2)_* \hat{\otimes} (Z_2)_*)^*$
with  completely bounded norm $ \leq {\lVert T_1 \rVert}_{cb} {\lVert T_2 \rVert}_{cb}$.
As in the last paragraph we can identify $((Y_k)_* \hat{\otimes} (Z_k)_*)^*$ with the normal Fubini
tensor product of $Y_k$ and $Z_k$.   Restricting 
 $u$ to the copy of $Y_1  \overline{\otimes} Z_1$, 
we get a completely bounded weak* continuous map from  $Y_1  \overline{\otimes} Z_1 \to Y_2  \overline{\otimes} Z_2$.

\begin{theorem} \label{extt}
The weak$^*$-exterior tensor product of $w^*$-rigged modules $Y$ and $Z$  is a $w^*$-rigged module.
\end{theorem}
\begin{proof}
Suppose that $\phi_{\alpha} : Y \to C_{n(\alpha)}(M) $ and $\psi_{\alpha}: C_{n(\alpha)}(M) \to Y$ are factorization maps for $Y$,   
and suppose that $\zeta_{\beta}: Z \to C_{m(\beta)}(N) $ and $\eta_{\beta}: C_{m(\beta)}(N) \to Z$ are factorization nets for $Z$, as in Definition \ref{wrig}.  By operator space theory we know that $C_n(M) \overline{\otimes} C_m(N) \cong C_{nm}(M \overline{ \otimes } N)$ completely isometrically and weak*-homeomorphically.  By functoriality of $\overline{\otimes}$, we can 
define $\phi_{\alpha} \otimes \zeta_{\beta} $ and $\psi_{\alpha} \otimes \eta_{\beta}$ of
$Y \overline{\otimes} Z$ through spaces $ C_{n(\alpha)}(M) \overline{\otimes}  C_{m(\beta)}(N)
\cong C_{n(\alpha) m(\beta)}(M \overline{ \otimes } N)$, and check that the conditions of Definition \ref{wrig} are met. 
\end{proof}

\begin{corollary} 
Suppose that $Y_1$ and $Z_1$ are right $w^*$-rigged modules over $M$,
that $Y_2$ and $Z_2$ are right $w^*$-rigged modules over $N$.
Suppose  that  $T_k : Y_k \to Z_k$ are  completely bounded and weak* continuous  module 
maps over $M$ and $N$ respectively, for $k = 1,2$.
Then $T_1\otimes T_2: Y_1  \overline{\otimes} Z_1 \to Y_2  \overline{\otimes} Z_2$  defines a unique completely bounded weak* continuous $(M \overline{\otimes} N)$-module map. Moreover, ${\lVert T_1 \otimes T_2 \rVert}_{cb} \leq {\lVert T_1 \rVert}_{cb} {\lVert T_2 \rVert}_{cb}$.
\end{corollary}   

\begin{proof}   Nearly all of this is just  the functoriality discussed above Theorem \ref{extt}.  It is easy to argue by weak* density arguments that $T_1\otimes T_2$  is a
$(M \overline{\otimes} N)$-module map.  
 \end{proof}

One may check that the weak$^*$-exterior tensor product has other properties
analogous to the interior  tensor product.   For example it is associative,
 `injective', and is appropriately projective for $w^*$-orthogonally 
complemented submodules and  
commutes with direct sums (we will prove this at the end of the next section).

\section{Complemented submodules} \label{cosub}

We say that a $w^*$-rigged module $Z$ over a dual operator algebra $M$
 is the $w^*${\em -orthogonal direct sum} of weak* closed submodules $Y$ and $W$, if  $Y + W = Z$,
$Y \cap W = (0)$, and $W$ and $Y$ are the ranges of 
two completely contractive idempotent maps $P$ and $Q$.   We say that $Y$ is $w^*${\em -orthogonally 
complemented} in $Z$ if there exists such a $W$. It follows from algebra
that the latter two maps $P, Q$ are unique, and are $M$-module maps
adding to $I_Z$ with $P Q = QP = 0$.   Also, they are weak* continuous.  Indeed suppose that $x_t = y_t + w_t$ is a bounded net with weak* limit 
$x = y + w$, where $y_t, y \in Y,$ and $w_t , w \in W$.   Then $(y_t)$ is bounded, and if $y_{t_\nu} \to z$
is a weak* convergent subnet, then $z \in Y$, and $w_{t_\nu} \to x - z \in W$.   It follows that 
$z = y$ and $x - z = w$.   By topology $y_t \to y$ weak*, so by  the Krein-Smulian theorem
$P$ is weak* continuous.   It follows from e.g.\ 
\cite[Theorem 7.2]{DB2} that $Z$ is the $w^*$-rigged module   
direct sum $Y \oplus^{c} W$ completely isometrically and unitarily. 
From Section 3.5 in \cite{BK2},  we see that the 
$w^*$-orthogonally complemented submodules of a
$w^*$-rigged module $Z$ are precisely the  ranges
 of completely contractive idempotents
 in $\mathbb{B}(Z)$.

\begin{proposition} \label{lmi}  The right $M$-summands in 
a $w^*$-rigged module $Z$ in the sense of
{\rm \cite{BEZ}} (see also {\rm \cite[Sections 4.5 and 4.8]{DBbook}}),
  are precisely the $w^*$-orthogonally complemented submodules of $Z$.
For any such submodule $W$ of $Z$ there is
a unique contractive linear projection from $Z$ onto $W$.
 \end{proposition}   

\begin{proof}   This is similar to the 
proof of Corollary  \ref{lms}, but using the 
fact from \cite[Theorem 2.3]{BK2} that 
the left multiplier operator algebra of $Z$ is
$\mathbb{B}(Z)$, so that the orthogonal projections here
 are the completely contractive idempotents
 in $\mathbb{B}(Z)$.
\end{proof} 
 
\begin{example}  Unlike the case when $M$ is a von Neumann algebra 
(see e.g.\ {\rm 8.5.16} in {\rm \cite{DBbook}}), weak* closed submodules of $w^*$-rigged modules
(or even of weak* Morita equivalence bimodules), need not be $w^*$-orthogonally complemented.
For example if $f$ is a nontrivial  inner function in $M = H^{\infty}(\mathbb{D})$ (such as the monomial 
$z$) then 
$Y = f H^{\infty}(\mathbb{D})$ is not complemented in the 
$M$-module $Z = H^{\infty}(\mathbb{D})$.  We note that $Y$ is a weak* Morita equivalence bimodule, with $\tilde{Y} = f^{-1} M$.   The latter is not a subset of $M$, and indeed the adjoint $\tilde{i}$ 
of the inclusion map $i : Y \to Z$ is not a projection.    \end{example}

\begin{lemma}  \label{nned}
Let $Z$ be a $w^*$-rigged module over $M$ and 
let $P:Z  \to Z$ be a $w^*$-continuous completely contractive idempotent
module map. Then the range of $P$ is a $w^*$-rigged module over $M$,
which is $w^*$-orthogonally complemented in $Z$. Also
$P$ is adjointable both as a map into $Z$ and  into $P(Z)$.
The dual module $\widetilde{P(Z)}$ of $P(Z)$ can be identified completely isometrically and $w^*$-homeomorphically
with the weak$^*$-orthogonally complemented
submodule $\tilde{P}(\tilde{Z})$ of $\tilde{Z}$, with the dual pairing
being the restriction of the pairing $\tilde{Z} \times Z \to M$.  
\end{lemma}

\begin{proof}
It is easy to see from the Remark after 
Theorem 2.7 in \cite{BK2}, and considering 
the maps between $Z$ and $Y = P(Z)$,
that $Y$ is a $w^*$-rigged module over $M$.  By Proposition 3.4 in \cite{BK2}, $P$ is adjointable both as a map into $Y$ and $Z$.    Since $P$ 
is an orthogonal projection in $\mathbb{B}(Z)$, $Y$ is $w^*$-orthogonally 
complemented in $Z$ (cf.\  \cite[Theorem 3.9]{BK2}).      
 We define $W$ =  Ran$(\tilde{P})$ = $\{f \circ P : f \in \tilde{Z} \}$. This is easily seen
to be a weak* closed submodule of $\tilde{Y}$.    Note that $CB^\sigma(Y,M) = 
\{ f_{\vert Y} : f \in CB^\sigma(Z,M) \}$.   The map $f \mapsto  f_{\vert Y} : W  \to CB^\sigma(Y,M)$
is a complete isometric $M$-module map, so that $\tilde{Y} \cong \tilde{P}(\tilde{Z})$.   The 
remaining assertion is now easy to check.
 \end{proof}

\begin{proposition}  If $Y$ is a weak$^*$-orthogonally complemented submodule
in a $w^*$-rigged module  $Z$, then $\mathbb{B}(Y)$ is a completely isometrically isomorphic
to a weak* closed  completely contractively weak* complemented subalgebra
of $\mathbb{B}(Z)$.
\end{proposition}

\begin{proof}     Let $i : Y \to Z$ be the inclusion and $P : Z \to Y$ the projection.
Then by functoriality of the tensor product,
$$i \otimes \tilde{P} : Y \otimes^{\sigma h}_M \tilde{Y} = \mathbb{B}(Y)
\to Z \otimes^{\sigma h}_M \tilde{Z} = \mathbb{B}(Z)$$ is completely contractive  and weak* continuous, and
is easy to check is a 
homomorphism.  Similarly one obtains a completely contractive weak* continuous retraction
$P \otimes \tilde{i} : Z \otimes^{\sigma h}_M \tilde{Z} = \mathbb{B}(Z) \to Y \otimes^{\sigma h}_M \tilde{Y} = \mathbb{B}(Y)$, with $(P \otimes \tilde{i}) \circ (i \otimes \tilde{P}) = I$.  
 \end{proof}

For the following result we recall that the $W^*$-dilation of a right $w^*$-rigged module $Z$ over a dual operator algebra $M$ is the canonical right $W^*$-module over  a von Neumann algebra $N$ generated by $M$
given by $Y \otimes_\theta N$.    Here $\theta : M \to N$ is the inclusion.  

\begin{corollary}
Let $Z$ be a right $w^*$-rigged module over a dual operator algebra $M$, and suppose 
that $Y$ is a subspace of $Z$, with $i :  Y \to Z$ the inclusion map. The following are equivalent:
\begin{enumerate} 
\item $Y$ is weak$^*$-orthogonally complemented in $Z$. 
\item $Y$ is a  $w^*$-rigged module over $M$ and there exists a completely contractive 
weak* continuous $M$-module map $j : \tilde{Y} \to \tilde{Z}$ 
such that $\tilde{i} \circ j = I_{\tilde{Y}}$.  
\item  $Y$ is a  $w^*$-rigged module over $M$, and there is a von Neumann algebra $N$ generated by $M$
such that for the induced 
map $i \otimes I_N$ between the $W^*$-dilations of $Y$ and $Z$ with respect to $N$,  we have that $i \otimes I_N$ is 
an isometry whose $W^*$-module adjoint $(i \otimes I_N)^*$ maps $Z \otimes 1_N$ into $Y \otimes 1_N$.  
\item Same as {\rm (3)}, but for every von Neumann algebra $N$ generated by $M$.
\end{enumerate}
\end{corollary}

\begin{proof}          (1) $\Rightarrow$ (4) \ If $P : Z \to Y$ is the projection
then we have adjointable contractions  $f = (i \otimes I_N) : Y \otimes^{\sigma h}_M N \to Z\otimes^{\sigma h}_M N$
and $g = (P \otimes I_N) : Z \otimes^{\sigma h}_M N \to Y \otimes^{\sigma h}_M N$ with 
$g \circ f = I$.    It follows that $f$ is an isometry, $g = f^*$, and $f^* = g$  maps $Z \otimes 1_N$ into $Y \otimes 1_N$.

 (3) $\Rightarrow$ (1) \  Let $j$ be the canonical isometry from $Y$ into  its  $W^*$-dilation,
which is a complete isometry by 3.4 in \cite{BK2}.
 It follows that $W = (i \otimes I_N) (Y \otimes^{\sigma h}_M N)$ is a 
weak* closed submodule of $Z\otimes^{\sigma h}_M N$, and the latter is a $W^*$-module.   By  e.g.\ {\rm 8.5.16} in {\rm \cite{DBbook}}, 
the $C^*$-module adjoint  of  $i \otimes I_N$ is a contractive weak* continuous projection $P$ from $Z\otimes^{\sigma h}_M N$
onto $W$.   Thus  $P \circ (i \otimes I_N) = I$ on $Y \otimes^{\sigma h}_M N$.    
Define $Q(z) = j^{-1} (i \otimes I_N)^{-1} P(z \otimes 1)$, this is a weak* continuous completely contractive $M$-module projection  onto $Y$.     Indeed $$Q(i(y)) =   j^{-1} (i \otimes I_N)^{-1} P(i(y) \otimes 1)
=  j^{-1} (y \otimes 1) = y, \qquad y \in Y.$$

Clearly (4)  implies  (3), and (1)  implies all the others.  

(2) $\Rightarrow$ (1) \ $P = j \circ \tilde{i}$  is a  weak* continuous 
completely contractive projection onto $j(\tilde{Y})$.   So the latter is  weak$^*$-orthogonally complemented in $\tilde{Z}$.  Hence by Lemma \ref{nned} its dual module may be identified with $\tilde{P}(Z) =
i(\tilde{j}(Z)) = Y$ (note that $\tilde{j} \circ i = I_Y$, so $\tilde{j}$ maps onto $Y$), and this is  weak$^*$-orthogonally complemented in $Z$.
\end{proof}

\begin{remark}     {\rm It seems possible that the equivalences in the last result still hold with some of the words
`weak* continuous'  or `$M$-module' removed in (2).   However this seems quite difficult at this
present time, although the last assertion of Proposition
\ref{lmi} seems pertinent here.
Things are better if $Z$ is a module of the kind considered 
in Proposition  \ref{bmex}.
If we are in that case, suppose 
that  there exists a completely contractive
 $M$-module map $j : \tilde{Y} \to \tilde{Z}$
such that $\tilde{i} \circ j = I_{\tilde{Y}}$ as in (2).  
Then $P = j \circ \tilde{i}$ is a
completely contractive $M$-module projection 
on $\tilde{Z} = CB^\sigma(Z,M)_M = CB(Y,M)_M$. 
  Hence it is weak* continuous 
by Proposition  \ref{bmex}, and we can continue as in the 
proof of (2) $\Rightarrow$ (1) above.}
\end{remark}  

At the end of Section 3 in \cite{BK2} we mentioned  with a sketchy proof the fact that direct sums commute 
with the interior tensor product; indeed we have left and right distributivity of 
$\otimes^{\sigma h}_M$ over column direct sums of $w^*$-rigged
modules.    It is also true that  direct sums commute 
with the exterior tensor product.   The proof we give of the latter fact will cover the  interior tensor product cases too, or 
is easily adaptable to those.

\begin{proposition} \label{comm}   Suppose that $M, N$ are dual operator algebras.
If $(Y_k)_{k \in I}$ is a family of  right $w^*$-rigged modules
over $M$, and  $Z$ is a right $w^*$-rigged module over $N$ then we have  $$(\oplus^c_k \; Y_k) \bar{\otimes} Z \cong
\oplus^c_k \; (Y_k  \bar{\otimes} Z) ,$$ unitarily as right $w^*$-rigged modules.   \end{proposition}  

\begin{proof}    We shall prove the more general statement that $$(\oplus^c_k \; Y_k) \otimes_\beta Z \cong
\oplus^c_k \; (Y_k  \otimes_\beta Z) ,$$ unitarily as right $w^*$-rigged modules, where 
$\otimes_\beta$ is any functorial tensor product (that is,  the tensor product of weak* continuous completely contractive
right module maps is also a weak* continuous completely contractive
right module map), that produces a right $w^*$-rigged module from 
right $w^*$-rigged modules, and for which the canonical map $Y \times Z \to Y  \otimes_\beta Z$ is
separately weak* continuous and has range whose span is weak* dense.   These are true for 
the interior and  exterior tensor product (see \cite{BK1} particularly Section 2 there), and e.g.\ 1.6.5 in 
\cite{DBbook}).  

We will use the functoriality of $\otimes_\beta$ and Theorem 3.9 in \cite{BK2}: If $Y = \oplus^c_k \; Y_k$
and $i_k, \pi_k$ are as in that result, then $\pi_k \otimes I$ and $i_k  \otimes I$ are weak* continuous completely  contractive right module maps that  
compose to the identity on $Y_k  \otimes_\beta Z$ (since their composition is weak* continuous and equals
$I$ on the weak* dense subset $Y \otimes Z$).  Similarly, they also satisfy $(\pi_k \otimes I) (i_j  \otimes I)  = 0$ 
if $j \neq k$.   Thus we will be done by  Theorem 3.9 in \cite{BK2} if $\sum_k \, (i_k \otimes I) (\pi_k  \otimes I)  = I$
in the weak* topology of $\mathbb{B}(Y  \otimes_\beta Z)$.     To see this, let $T_\Delta = \sum_{k \in \Delta}
\, i_k \pi_k$ for finite $\Delta \subset I$.    We will be done if $T_\Delta \otimes I \to I$ weak* in $\mathbb{B}(Y  \otimes_\beta Z)$, since $T_\Delta \otimes I  = \sum_{k \in \Delta} \, (i_k \otimes I) (\pi_k  \otimes I)$.

Indeed we shall prove a more general fact, that if a bounded net $S_t \to S$ weak* in $\mathbb{B}(Y)$ then 
$S_t \otimes I \to S \otimes I$ weak* in $\mathbb{B}(Y  \otimes_\beta Z)$.    Suppose that we have a 
weak* convergent subnet $S_{t_\nu}  \otimes I \to R$.   By Theorem 3.5 in \cite{BK2} we have that 
$R \in \mathbb{B}(Y  \otimes_\beta Z)$, and hence $R$ is weak* continuous.   For $y \in Y, z \in Z$ we 
have $S_t(y) \to S(y)$ weak* (this follows since  the latter describes the weak* convergence of bounded nets in 
$CB(Y)$ by e.g.\ 1.6.1 in \cite{DBbook}, and 
by  \cite[Theorem 2.3]{BK2} $\mathbb{B}(Y)$ is a weak* closed subalgebra of $CB(Y)$).  Hence  $$(S_t \otimes I)(y \otimes z) = S_t(y)  \otimes  z \to S(y)  \otimes z
= (S  \otimes I)(y \otimes z) .$$
Thus $R(y \otimes z) =  (S  \otimes I)(y \otimes z)$.  Hence $R =  S  \otimes I$ since they are
weak* continuous and agree on a dense subset.    By topology it follows that 
$S_t \otimes I \to S \otimes I$ weak* as desired.  
\end{proof}

\subsection*{Acknowledgements}

 The author thanks Upasana Kashyap for several discussions, and for prompting him 
to persevere with this project.

\end{document}